\documentclass[12pt]{amsart}

\usepackage{amssymb}
\usepackage{graphics,color,epic}
\usepackage[all,2cell,ps,dvips]{xy}

\def\ca{{\mathcal A}}
\def\rca{{\mathcal A}^\prime}
\def\cca{\ca^{\mathbb C}}
\def\rcca{{\rca}^{\mathbb C}}

\def\codim{{\mathrm{codim}}}
\def\fs{{\mathfrak{S}}}

\DeclareMathOperator{\br}{br}
\DeclareMathOperator{\ao}{ao}

\def\crs{{\circlearrowleft}}
\def\Sn{{\fs_n}}

\def\Hz{{\hat{0}}}
\def\CL{{\mathcal{C}^\downarrow}}
\def\aell{{\ell^\prime}}
\def\nn{{\mathbb N}}
\def\rr{{\mathbb R}}
\def\cR{{\rm re}}
\def\INV{{\rm {INV}}}

\def\GL{{\rm GL}}
\def\SL{{\rm SL}}

\newcommand{\wh}{\widehat}
\newcommand{\BG}{\mathrm{BG}}

\definecolor{light}{gray}{.55}
\definecolor{dark}{gray}{.35}

\newtheorem{theorem}{Theorem}[section]
\newtheorem{lemma}[theorem]{Lemma}
\newtheorem{proposition}[theorem]{Proposition}
\newtheorem{corollary}[theorem]{Corollary}

\newtheorem{definition}[theorem]{Definition}

\newtheorem{open}[theorem]{Open problem}
\newtheorem{observation}[theorem]{Observation}

\begin{document}

\title[A conjecture of Postnikov]{From Bruhat intervals to 
  intersection lattices and a conjecture of Postnikov}

\author[A. Hultman]{Axel Hultman} 
\address{Department of Mathematics, KTH-Royal Institute of Technology, 
  SE-100 44, Stockholm, Sweden.}
\author[S. Linusson]{Svante Linusson} 
\address{Department of Mathematics, KTH-Royal Institute of Technology, 
  SE-100 44, Stockholm, Sweden.}
\author[J. Shareshian]{John Shareshian}
\address{Department of Mathematics\\
         Washington University\\
         St Louis, MO 63130\\ USA. }
\author[J. Sj\"ostrand]{Jonas Sj\"ostrand}
\address{Department of Mathematics and Physics, M\"alardalen University,
Box 883, SE-721 23, V\"aster{\aa}s, Sweden.}
\email{hultman@math.kth.se}
\email{linusson@math.kth.se}
\email{shareshi@math.wustl.edu}
\email{jonass@kth.se}

\date{October 5, 2007}

\begin{abstract}
We prove the conjecture of A.\ Postnikov that ({\rm A})
the number of regions in the inversion hyperplane arrangement
associated with a permutation $w\in \Sn$ is at most the number of
elements below $w$ in the Bruhat order, and ({\rm B}) that equality
holds if and only if $w$ avoids the patterns $4231$, $35142$, $42513$
and $351624$. Furthermore, assertion ({\rm A}) is extended to all
finite reflection groups.

A byproduct of this result and its proof is a set of inequalities
relating Betti numbers of complexified inversion arrangements to Betti
numbers of closed Schubert cells. Another consequence is a simple
combinatorial interpretation of the chromatic polynomial of the
inversion graph of a permutation which avoids the above patterns. 
\end{abstract}

\maketitle

\section{Introduction}

We confirm a conjecture of A. Postnikov \cite[Conj 24.4(1)]{Po}, 
relating the interval below
a permutation $w \in \fs_n$ in the Bruhat order and a hyperplane
arrangement determined by the inversions of $w$.  Definitions of
key objects discussed but not defined in this introduction can be found
in Section \ref{pre}.

Fix $n \in \nn$ and $w \in \fs_n$.  An {\it inversion} of $w$ is
 a pair $(i,j)$ such that $1 \leq i<j \leq n$ and $iw<jw$. (Here 
we write $w$ as a function acting from the right on 
$[n]:=\{1,\dotsc,n\}$.)  We write $\INV(w)$ for the set of inversions of $w$.

For $1 \leq i<j \leq n$, set
\[
H_{ij}:=\{(v_1,\dotsc,v_n) \in \rr^n\mid v_i=v_j\},
\]
so $H_{ij}$ is a hyperplane in $\rr^n$.  Set
\[
\rca_w:=\{H_{ij}\mid (i,j) \in \INV(w)\},
\]
so $\rca_w$ is a central hyperplane arrangement in $\rr^n$.  
Let $\cR(w)$ be the number of connected components 
of $\rr^n \setminus \cup \rca_w$. Let $\br(w)$ be the size
 of the ideal generated by $w$ in the Bruhat order on $\fs_n$.

The first part of Postnikov's conjecture is that
\begin{itemize}
\item[(A)] for all $n \in \nn$ and all $w \in \fs_n$ we
 have $\cR(w) \leq \br(w)$.
\end{itemize}
In Theorem \ref{th:injective} below, we give a generalization 
of (A) that holds for all finite reflection groups.

Let $m \leq n$, let $p \in \fs_m$ and let $w \in \fs_n$.  
We say $w$ {\it avoids} $p$ if there do not 
exist $1 \leq i_1<i_2<\dotsb<i_m \leq n$ such that for
 all $j,k \in [m]$ we have $i_jw<i_kw$ if and only if $jp<kp$.  
The second part of Postnikov's conjecture is that
\begin{itemize}
\item[(B)] for all $n \in \nn$ and all $w \in \fs_n$, 
we have $\br(w)=\cR(w)$ if and only if $w$ avoids all 
of 4231, 35142, 42513 and 351624.
\end{itemize}
Here we have written the four permutations to be avoided 
in {\it one line notation}, that is, we write $w \in \fs_n$ as
 $1w\dotsm nw$.  As is standard, we call the permutations to be 
avoided {\em patterns}.  With Theorem \ref{th:Avoiding} we show that
 avoidance of the four given patterns is necessary for the equality 
of $\br(w)$ and $\cR(w)$ and with Corollary \ref{co:surjective} we show 
that this avoidance is sufficient, thus proving all of Postnikov's conjecture.

We remark that the avoidance of the four given patterns has
 arisen in work of Postnikov on total positivity (\cite{Po}), 
work of Gasharov and Reiner on Schubert varieties in partial flag
 manifolds (\cite{GR}) and work of Sj\"ostrand (\cite{sjostrand}) on
 the Bruhat order. In Section~\ref{sec:axel}, we give
yet another characterization of the permutations that avoid these patterns.

The Bruhat order (on any Weyl group) describes the containment 
relations between the closures of Schubert cells in the associated
 flag variety (see for example \cite{chevalley, fulton}).  
Inequality (A) (along 
with our proof of it) indicates that there might be some relationship between
 the cohomology of the closure of the Schubert cell indexed
 by $w$ and the cohomology of the complexification of the
 arrangement $\rca_w$.  In Proposition \ref{pr:betprop} we provide three
 inequalities relating these objects when $w \in \fs_n$ avoids the
 four patterns mentioned above.

In Section~\ref{sec:chromo}, we show how the chromatic polynomial
of the inversion graph of $w\in\fs_n$ (or, equivalently,
the characteristic polynomial of $\rca_w$) keeps track of
the transposition distance from $u$ to $w$ for $u\le w$ in Bruhat order.
In Section~\ref{sec:example} we provide an example to illustrate what
our results say about a specific permutation, and in Section~\ref{sec:open} 
we list some open problems.

\noindent
{\bf Acknowledgement:} Linusson is a Royal Swedish Academy of Sciences 
Research Fellow supported by a grant from the Knut and Alice Wallenberg 
Foundation.  Shareshian is supported by NSF grant DMS-0604233.

\section{Prerequisites} \label{pre}

In this section, we review basic material on hyperplane arrangements
and Coxeter groups that we will use in the sequel. For more information on
these subjects the reader may consult, for example, \cite{stanley} and \cite{BB},
respectively.

A {\em Coxeter group} is a group $W$ generated by a finite set $S$ of
involutions subject only to relations of the form
$(ss^\prime)^{m(s,s^\prime)}=1$, where $m(s,s^\prime)=m(s^\prime,s)\geq
2$ if $s\neq s^\prime$. The pair $(W,S)$ is referred to as a
{\em Coxeter system}. 

The {\em length}, denoted $\ell(w)$, of $w\in W$ is the smallest $k$
such that $w=s_1\dotsm s_k$ for some $s_1, \dotsc, s_k \in S$. If
$w=s_1\dotsm s_k$ and $\ell(w) = k$, then the sequence $s_1\dotsm s_k$
is called a {\em reduced expression} for $w$. 

Every Coxeter group admits a partial order called the {\em Bruhat
  order}.

\begin{definition}\label{de:Bruhat}
  Given $u,w\in W$, we say that $u\leq w$ in the Bruhat order if every
  reduced expression (equivalently, some reduced expression) 
for $w$ contains a subword representing $u$. In
  other words, $u\leq w$ if whenever $w = s_1\dotsm s_k$ with each $s_i\in
  S$ and $\ell(w)=k$, there exist $1\leq i_1<\dotsb <i_j\leq k$ such that
  $u=s_{i_1}\dotsm s_{i_j}$. 
\end{definition}

Although it is not obvious from Definition \ref{de:Bruhat}, the Bruhat
order {\em is} a partial order on $W$. Observe that the identity
element $e\in W$ is the unique minimal element with respect to this order. 

Given $u,w\in W$, the
definition is typically not very useful for determinining whether
$u\leq w$. When $W=\Sn$ is a symmetric group, with $S$ being the set
of adjacent transpositions $(i\,\,\,i+1)$, the following nice criterion 
exists. For a permutation $w\in \Sn$ and $i,j\in [n]=\{1, \dotsc,
n\}$, let 
\[
w[i,j] = |\{m\in [i] \mid mw \geq j\}|.
\]
Let $P(w)=(a_{ij})$ be the permutation matrix corresponding to $w\in \Sn$ (so
$a_{ij}=1$ if $iw=j$ and
$a_{ij}=0$ otherwise). Then $w[i,j]$ is simply the number of ones
 weakly above and
weakly to the right of position $(i,j)$ in $P(w)$, that is, the 
number of pairs $(k,l)$ such that $k \leq i$, $j \leq l$ and $a_{kl}=1$.

A proof of the next proposition can be found in \cite{BB}.
\begin{proposition}[Standard criterion]\label{pr:standardcrit}
  Given $u,v\in \Sn$, we have $u\leq w$ in the Bruhat order if and only if
  $u[i,j]\leq w[i,j]$ for all $(i,j) \in [n]^2$.
\end{proposition}
In fact, it is only necessary to compare $u[i,j]$ and $w[i,j]$ for
certain pairs $(i,j)$; see Lemma \ref{lem:bubbles} below. 

Each finite Coxeter group $W$ can be embedded in some $\GL_n(\rr)$ in
 such a way that the elements of $S$ act as reflections.  
That is, having fixed such an embedding, for each $s \in S$ there is 
some hyperplane $H_s$ in $\rr^n$ such that $s$ acts on $\rr^n$ by
 reflection through $H_s$.  Thus a {\em reflection} in $W$ is defined to
 be an element conjugate to an element of $S$. Letting $T$ denote the set 
of reflections in $W$, we therefore 
have $T=\{w^{-1}sw\mid s\in S, \, w\in W\}$. Every finite subgroup 
of $\GL_n(\rr)$ generated by reflections is a Coxeter group.
A {\em natural geometric representation} of a Coxeter 
group $W$ is an embedding of the type just described in which no point 
in $\rr^n \setminus \{0\}$ is fixed by all of $W$.

Sometimes we work with the generating set $T$ rather than $S$.  We
 define the {\em
  absolute length} $\ell^\prime(w)$ as the smallest number of
reflections needed to express $w\in W$ as a product. In the case of
finite Coxeter groups, i.e.\ finite reflection groups, a nice formula
for the absolute length follows from work of Carter \cite[Lemma
2]{carter}. 

\begin{proposition}[Carter \cite{carter}] \label{pr:Carter}
  Let $W$ be a finite reflection group in a natural geometric
  representation. Then, the absolute length of $w\in W$ 
  equals the codimension of the space of fixed points of $w$.
\end{proposition}

Next, we recall a convenient interaction between reflections and (not
necessarily reduced) expressions. For a proof, the reader may consult
\cite[Theorem 1.4.3]{BB}. By a $\wh{\mbox{hat}}$ over an element, we
understand deletion of that element.

\begin{proposition}[Strong exchange property] \label{pr:SEP}
Suppose $w = s_1\dots s_k$ for some $s_i\in S$. If $t\in T$ has the
property that $\ell(tw)<\ell(w)$, then $tw = s_1\dots \wh{s_i}\dots
s_k$ for some $i\in [k]$.
\end{proposition}

A {\em real hyperplane arrangement} is a set $\ca$ of affine hyperplanes 
in some real vector space $V \cong \rr^n$.  We will assume that $\ca$ is
 finite.  The arrangement $\ca$ is called {\em linear} if each $H \in \ca$ 
is a linear subspace of $\rr^n$.  The {\em intersection lattice} of a
 linear arrangement $\ca$ is the set $L_\ca$ of all subspaces of $V$ 
that can be obtained by intersecting some elements of $\ca$, ordered by
 reverse inclusion.  (The minimal element $V$ of $L_\ca$ is obtained by
 taking the intersection of no elements of $\ca$ and will be denoted 
by $\Hz$.)  

A crucial property of $L_\ca$ is that it admits a so-called
{\em EL-labelling}. The general definition of such labellings is not
important to us; see \cite{bjorner} for details. Instead, we
focus on the properties of a particular EL-labelling of $L_\ca$, the {\em
  standard} labelling $\lambda$, which we now describe.

Let $\lhd$ denote the covering relation of $L_\ca$. Choose some total
ordering of the hyperplanes in $\ca$. To each covering $A\lhd
B$ we associate the label 
\[
\lambda(A\lhd B) = \min \{H\in \ca\mid H\leq B\text{ and }H\not \leq A\}.  
\]

The complement $V\setminus \cup \ca$ of the arrangement $\ca$ is a
disjoint union of contractible connected components called the {\em
  regions} of $\ca$. The number of regions can be computed from
$\lambda$. Given any saturated chain 
$C=\{A_0\lhd \dotsb \lhd A_m\}$ in $L_\ca$, say that $C$ is
{\em $\lambda$-decreasing} if $\lambda(A_{i-1}\lhd
A_i)>\lambda(A_i\lhd A_{i+1})$ for all $i\in [m-1]$. 

\begin{proposition}[Bj\"orner \cite{bjorner}, Zaslavsky
  \cite{zaslavsky}] \label{pr:regions}
  The number of regions of $\ca$ equals the number of
  $\lambda$-decreasing saturated chains that contain $\Hz$.
  \begin{proof}
    It follows from the theory of EL-labellings \cite{bjorner} that
    the number of chains with the asserted properties is
    \[
    \sum_{A\in L_\ca}|\mu(\Hz,A)|,
    \]
    where $\mu$ is the M\"obius function of $L_\ca$. By a result of
    Zaslavsky \cite{zaslavsky}, this number is precisely the number of
    regions of $\ca$.
  \end{proof}
\end{proposition}

Given a finite Coxeter group $W$ we may associate to it the {\em
  Coxeter arrangement} $\ca_W$. This is the collection of 
hyperplanes that are fixed by the various reflections in $T$ when
we consider $W$ as a finite reflection group in a standard geometric
representation. The isomorphism type of $L_\ca$ does not depend on the
choice of standard representation.

\section{From intersection lattices to Bruhat intervals} \label{sec:inj}

Let $(W,S)$ be a finite Coxeter system. Fix a reduced expression
$s_1\dotsm s_k$ for some $w\in W$. Given $i\in [k]$, define the
reflection 
\[
t_i = s_1\dotsm s_{i-1}s_is_{i-1}\dotsm s_1 \in T.
\]
The set $T_w = \{t_i\mid i\in [k]\}$ only depends on $w$ and not on the
chosen reduced expression. In fact, $T_w = \{t\in T\mid
\ell(tw)<\ell(w)\}$. We call $T_w$ the {\em inversion set} of 
$w$. If $W=\Sn$ and $T$ is the set of transpositions, then the
transposition $(i\, j)$ lies in $T_w$ if and only if $(i,j) \in \INV(w)$.  
Being reflections, the various $t_i$ correspond to reflecting
hyperplanes $H_i$ in a standard geometric representation of
$W$. Thus, $w$ determines an arrangement of real linear hyperplanes
\[
\ca_w = \{H_i\mid i\in [k]\}
\]
which we call the {\em inversion arrangement} of $w$. It is a
subarrangement of the Coxeter arrangement $\ca_W$. 

Let us order the hyperplanes in $\ca_w$ by $H_1 > H_2 > \dotsb > H_k$. We denote by $\lambda$ the standard
EL-labelling of the intersection lattice $L_w = L_{\ca_w}$ induced by
this order. In particular, $\lambda$ depends on the choice of reduced
expression for $w$. 

Let $\CL$ be the set of $\lambda$-decreasing saturated chains in $L_w$
that include the minimum element $\Hz$. By Proposition
\ref{pr:regions}, $\CL$ is in bijection with the set of regions of
$\ca_w$. We will construct an injective map from 
$\CL$ to the Bruhat interval $[e,w]$.

Let $C = \{\Hz=X_0 \lhd X_1 \lhd \dotsb \lhd X_m\} \subset L_w$ be a saturated
chain. Suppose, for each $i \in [m]$, we have $\lambda (X_{i-1}\lhd X_i) = H_{j_i}$. Define 
\[
p(C) = t_{j_1}\dotsm t_{j_m} \in W.
\]

\begin{proposition}\label{pr:map}
  If $C\in \CL$, then $p(C)w \leq w$ in the Bruhat order. Thus,
$C\mapsto p(C)w$
  defines a map $\phi: \CL \to [e,w]$.
  \begin{proof}
    When $C = \{\Hz=X_0 \lhd X_1 \lhd \dotsb \lhd X_m\} \subset L_w$ is
    $\lambda$-decreasing, we have 
    \[
    p(C)w = \prod_{i\in [k]\setminus \{j_1, \dotsc, j_m\}}s_i.
    \] 
    Thus, $p(C)w$ can be represented by an expression which is a
    subword of the chosen reduced expression for $w$.
  \end{proof}
\end{proposition}
A full description of $\phi$ when $w=(142) \in \fs_4$ appears in 
Section~\ref{sec:example}.
In order to deduce injectivity of $\phi$, we need the following lemma.

\begin{lemma}\label{le:dimension}
  For every saturated chain $C= \{\Hz=X_0 \lhd X_1 \lhd \dotsb \lhd X_m\} 
\subset L_w$, we have $\aell(p(C)) = m$. 
  \begin{proof}
    We proceed by induction on $m$, the case $m=0$ being trivial. 

    By construction, $\aell(p(C)) \leq m$. Suppose, in order to deduce a
    contradiction, that the inequality is strict. The inductive hypothesis
    implies $\aell(p(C\setminus X_m)) = m-1$. Thus, $\aell(p(C)) =
    m-2$. We may therefore write $p(C) = t_1^\prime \dotsm
    t_{m-2}^\prime$ for some reflections $t_i^\prime \in T$ through
    corresponding hyperplanes $H_i^\prime$. 

    Recall the notation $\lambda(X_{i-1}\lhd X_i) = H_{j_i}$ with
    corresponding reflection $t_{j_i}$. Let $F$ denote the fixed point
    space of $p(C)t_{j_m} = p(C\setminus X_m)$. Then, $X_{m-1} =
    H_{j_1}\cap \dotsb \cap H_{j_{m-1}} \subseteq 
    F$. By Proposition \ref{pr:Carter}, $\codim (F) = \aell(p(C)t_{j_m}) =
    m-1 = \codim(X_{m-1})$. Thus, $F=X_{m-1}$. On the other hand,
    $p(C)t_{j_m} = t_1^\prime\dotsm t_{m-2}^\prime t_{j_m}$. Therefore,
    $F\supseteq H_1^\prime \cap \dotsb \cap H_{m-2}^\prime \cap
    H_{j_m}$. Now, $\codim(F)=m-1 \geq \codim (H_1^\prime \cap \dotsb \cap
    H_{m-2}^\prime \cap H_{j_m})$ so that, in fact, $F = H_1^\prime \cap 
    \dotsb \cap
    H_{m-2}^\prime \cap H_{j_m}$. Hence,
    $H_{j_m} \supseteq X_{m-1}$, which is impossible given the
    deefinition of $\lambda$.
  \end{proof}
\end{lemma}

We are now in position to prove the main result of this section.

\begin{theorem}\label{th:injective}
  The map $\phi:\CL \to [e,w]$ is injective.
  \begin{proof}
    If $C$ is the saturated chain $\Hz =X_0\lhd\dotsb\lhd X_m$ in
    $L_w$, then $X_m$ is contained in the fixed point space of $p(C)$
    (since $p(C)$ is a product of reflections through hyperplanes,
    all of which contain $X_m$). Lemma \ref{le:dimension} and Proposition
    \ref{pr:Carter} therefore imply that $X_m$ is the
    fixed point space of $p(C)$. In particular, if two chains have the
    same image under $p$, then their respective maximum elements coincide.

    Now suppose $p(C) = p(D)$ for some $C,D\in \CL$. We shall show that
    $C=D$. Write $C = \{\Hz
    =X_0\lhd \dotsb \lhd X_m\}$ and $D = \{\Hz=Y_0 \lhd \dotsb \lhd
    Y_{m^\prime}\}$. We have shown that 
    $m=m^\prime$ and $X_m = Y_m$. Since both $C$ and $D$
    are $\lambda$-decreasing, the construction of $\lambda$ implies
    $\lambda(X_{m-1}\lhd X_m) = \lambda(Y_{m-1}\lhd Y_m) = H$, where $H$ is the
    smallest hyperplane below $X_m = Y_m$ in $L_w$. With $t$ denoting the
    reflection corresponding to $H$, we thus have $p(C\setminus X_m) =
    p(D\setminus Y_m) = p(C)t = p(D)t$. Our theorem is proved by induction 
on $m$.
  \end{proof}
\end{theorem}

Let us explain how the first part of Postnikov's conjecture, statement
(A) in the Introduction, follows from Theorem \ref{th:injective}. The
symmetric group $\Sn$ acts on $\rr^n$ by permuting
coordinates. Under this action, the transposition $(i\, j)$ acts by a
reflection in the hyperplane given by $x_i=x_j$. However, this is not
quite a natural geometric representation of $\Sn$ because the entire
line given by $x_1=\dots =x_n$ is fixed by all elements. To rectify
the situation we may study the restriction of the action to the subspace
$V^{(n-1)} \subset \rr^n$ that consists of the points in $\rr^n$ whose
coordinates sum to zero. Thus, $\ca_w$ is a hyperplane arrangement in
$V^{(n-1)}$.

Recalling our convention that $uw$ means ``first
$u$, then $w$'' for $u,w\in \Sn$ we see that $(i\, j)\in T_w$ if and only if
$(i,j)$ is an inversion of $w$ in the ordinary sense. Thus, for $w\in
 \Sn$,
\[
\ca_w = \{H\cap V^{(n-1)}\mid H\in \rca_w\}.
\]
In the language of \cite{stanley}, $\ca_w$ is the {\em
  essentialization} of $\rca_w$. The regions in the
  complements of $\ca_w$ and $\rca_w$ are in an obvious bijective
  correspondence and statement (A) follows.

Although we do not know when $\phi$ is surjective for an arbitrary finite 
reflection group, for symmetric groups we have the following result, whose 
proof is contained in the nect two sections.
\begin{theorem}\label{th:main}
If $w\in\fs_n$, the map $\phi$ is surjective (and hence bijective)
if and only if $w$ avoids the patterns
4231, 35142, 42513 and 351624.
\end{theorem}

\section{A necessity criterion for surjectivity in symmetric groups}

We now confine our attention to the type $A$ case when $W = \Sn$ is a symmetric
group. Depending on what is most convenient, either one-line notation
or cycle notation is used to represent a permutation $w \in \Sn$.
In this setting, as we have seen, $T$ becomes the set of transpositions in 
$\Sn$ and $T_w = \{(i\, j)\mid i<j\text{ and }iw>jw\}$ can be
identified with $\INV(w)$.

\begin{theorem} \label{th:Avoiding}
Suppose $W$ is a symmetric group. If $\phi:\CL \to
  [e,w]$ is surjective, then $w$ 
  avoids the patterns $4231$, $35142$, $42513$ and $351624$.
  \begin{proof}
    It follows from Lemma \ref{le:dimension} that if $u \leq w$ is in the
    image of $\phi$, then $uw^{-1}$ can be written as a product of
    $\aell(uw^{-1})$ inversions of $w$. Below we construct, for $w$
    containing each of the four given patterns, elements $u\leq w$ that
    fail to satisfy this property.

\noindent    {\bf Case $4231$.}
    Suppose $w$ contains the pattern $4231$ in positions
    $n_1$, $n_2$, $n_3$, and $n_4$, meaning that $n_1w>n_3w>n_2w>n_4w$.
    Then, let $u=(n_1\, n_4)(n_2\, n_3)w$. Invoking the standard criterion,
    Proposition~\ref{pr:standardcrit}, it
    suffices to check $(1\, 4)(2\, 3)4231 = 1324 < 4231$ in order to
    conclude $u<w$. Now, $uw^{-1} = (n_1\, n_4)(n_2\, n_3)$ has absolute
    length $2$. However, $uw^{-1}$ cannot be written as a product of two
    inversions of $w$, because $(n_2\, n_3)$ is not an inversion. 

 \noindent   {\bf Case $35142$.}
    Now assume $w$ contains $35142$ in positions $n_1, \dotsc, n_5$.
    Define $u=(n_1\, n_3\, n_4)(n_2\, n_5)w$. Again we have $u < w$; this 
time since
    $(1\, 3\, 4)(2\, 5)35142 = 12435 < 35142$. We have
    $uw^{-1}=(n_1\, n_3\, n_4)(n_2\, n_5)$ which is of absolute length
    $3$. Neither $(n_1\, n_4)$ nor $(n_3\, n_4)$ is an inversion of $w$, so
    $u$ cannot be written as a product of three members of $T_w$.

\noindent    {\bf Case $42513$.}
    Next, suppose $w$ contains $42513$ in $n_1$ through $n_5$. 
    Then, we
    let $u=(n_2\, n_5\, n_3)(n_1\, n_4)w$ and argue as in the previous cases.

\noindent    {\bf Case $351624$.}
    Finally, if $w$ contains $351624$ in positions
    $n_1$ through $n_6$, we may use $u=(n_1\, n_3\, n_6\, n_4)(n_2\, n_5)w$
    and argue as before.
  \end{proof}
\end{theorem}

\section{Pattern avoidance implies $\br(w)=\cR(w)$}
Let $\hat{\fs}_n\subseteq\fs_n$
denote the set of permutations that avoid the four patterns 4231,
35142, 42513, 351624.

In this section we will represent permutations $\pi\in\fs_n$ by rook diagrams.
These are $n$ by $n$ square boards with a rook in entry $(i,j)$, i.e.~row $i$
and column $j$, if $i\pi=j$. If $x$ is a rook, we will write $x_i$ for
its row number and $x_j$ for its column number.

The \emph{inversion graph} of $\pi$, denoted by
$G_\pi$, is a simple undirected graph with the rooks as vertices and
an edge between two rooks
if they form an inversion of $\pi$, i.e.~if one of them is
south-west of the other one.
Let $\ao(\pi)=\ao(G_\pi)$ denote the number of acyclic orientations
of $G_\pi$. Note that
$\ao(\pi)$ equals the number of regions $\cR(\pi)$
of the hyperplane arrangement $\rca_\pi$.

Following Postnikov~\cite{Po}, we will call a permutation $\pi$
\emph{chromobruhatic} if $\br(\pi)=\ao(\pi)$. (A motivation for this
appellation is given in Section~\ref{sec:chromo}.)
Our goal in this section is to prove
that all $\pi\in\hat\fs_n$ are chromobruhatic. This will be accomplished
as follows: First we show that if $\pi$ (or its inverse)
has something called a \emph{reduction pair},
which is a pair of rooks with certain properties,
then there is a recurrence relation for $\br(\pi)$ in terms of $\br(\rho)$
for some permutations $\rho \in \hat{\fs}_n \cup \hat{\fs}_{n-1}$ that 
are ``simpler'' than $\pi$ in a sense that
will be made precise later. It turns out that the very same recurrence 
relation also
works for expressing $\ao(\pi)$ in terms of a few $\ao(\rho)$. Finally, we show
that every $\pi\in\hat\fs_n$ except the identity permutation has a
reduction pair, and hence $\br(\pi)=\ao(\pi)$ by induction.

We will need two useful lemmas about the Bruhat order on the symmetric group.
The first is a well-known variant of Proposition~\ref{pr:standardcrit}
(see e.g.~\cite{GR}).
A square that has a rook strictly to the left in the same
row and strictly below
it in the same column is called a \emph{bubble}.
\begin{lemma}\label{lem:bubbles}
  Let $\pi,\sigma\in\fs_n$. Then $\sigma\le\pi$ in the Bruhat order if and 
only if
  $\sigma[i,j]\le\pi[i,j]$ for every bubble $(i,j)$ of $\pi$.
\end{lemma}

If $\pi$ avoids the forbidden patterns, there is an even simpler
criterion. Define the \emph{right hull of $\pi$}, denoted by
$H_R(\pi)$, as the set of squares in the rook diagram of $\pi$
that have at least one rook weakly south-west of them and at least
one rook weakly north-east of them. Figure~\ref{fig:hull_example} shows
an example. The following lemma is due to Sj\"ostrand~\cite{sjostrand}.
\begin{figure}
\resizebox{20mm}{!}{\includegraphics{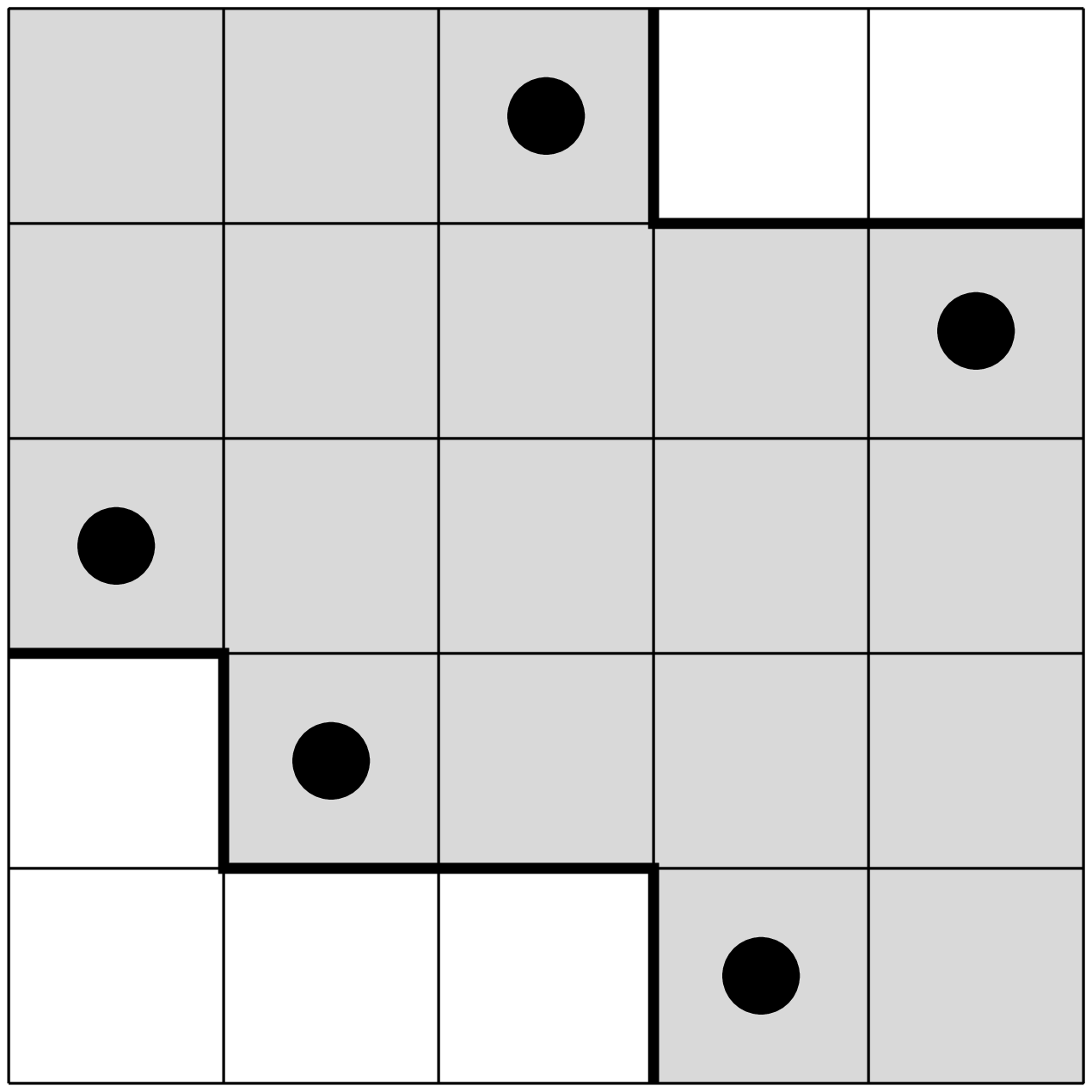}}
\caption{The shaded region constitutes the
right hull of the permutation 35124.}
\label{fig:hull_example}
\end{figure}
\begin{lemma}\label{lem:right_hull}
  Let $\pi\in\hat\fs_n$ and $\sigma\in\fs_n$.
  Then $\sigma\le\pi$ in the Bruhat order if and only if all rooks of
  $\sigma$ lie in the right hull of $\pi$.
\end{lemma}

For a permutation $\pi\in\fs_n$, the rook diagram of the inverse permutation
$\pi^{-1}$ is obtained by transposing the rook diagram of $\pi$. Define
$\pi^\crs=\pi_0\pi\pi_0$, where $\pi_0=n(n-1)\dotsm 1$
denotes the maximum element (in the Bruhat order) of $\fs_n$. Note
that the rook diagram of $\pi^\crs$ is obtained by a 180 degree rotation of
the rook diagram of $\pi$.
\begin{observation}\label{obs:trans_and_rot}
  The operations of transposition
  and rotation of the rook diagram of a permutation have the following 
properties.
  \begin{enumerate}
  \item[(a)] They are automorphisms of the Bruhat order, i.e.
    \[
    \sigma\le\tau\,\Leftrightarrow\,\sigma^{-1}\le\tau^{-1}
    \,\Leftrightarrow\,\sigma^\crs\le\tau^\crs
    \,\Leftrightarrow\,(\sigma^\crs)^{-1}\le(\tau^\crs)^{-1}.
    \]
  \item[(b)] They induce isomorphisms of inversion graphs, so
    \[
    G_\sigma\cong G_{\sigma^{-1}}\cong G_{\sigma^\crs}\cong 
G_{(\sigma^\crs)^{-1}}.
    \]
  \item[(c)] The set of the four forbidden patterns is closed under
    transposition and rotation, so
    \[
    \sigma\in\hat{\fs}_n\,\Leftrightarrow\,\sigma^{-1}\in\hat{\fs}_n
    \,\Leftrightarrow\,\sigma^\crs\in\hat{\fs}_n
    \,\Leftrightarrow\,(\sigma^\crs)^{-1}\in\hat{\fs}_n.
    \]
  \end{enumerate}
  From (a) and (b) it follows that
  $\sigma$, $\sigma^{-1}$, $\sigma^\crs$ and $(\sigma^\crs)^{-1}$
  are either all chromobruhatic or all non-chromobruhatic.
\end{observation}

If $x$ is a rook in the diagram of $\pi$ then the image of $x$ under any 
composition of transpositions and rotations is a rook in the diagram of 
the resulting permutation.  In what follows, we sometimes discuss properties 
that the image rook (also called $x$) has in the resulting diagram, while 
still thinking of $x$ as lying in its original position in the diagram of 
$\pi$.

\begin{definition}
  Let $\pi\in\fs_n$ and let $x,y$ be a pair of rooks that is a descent,
  i.e.~$y_i=x_i-1$ and $x_j<y_j$.
  Then, $x,y$ is a \emph{light reduction pair} if
  we have the situation in Figure~\ref{fig:reduction_pairs}(a), i.e.
  \begin{itemize}
  \item there is no rook $a$ with $a_i<y_i$ and $a_j>y_j$, and
  \item there is no rook $a$ with $a_i>x_i$ and $x_j<a_j<y_j$.
  \end{itemize}
  The pair $x,y$ is called a \emph{heavy reduction pair} if we have the 
  situation
  in Figure~\ref{fig:reduction_pairs}(b), i.e.
  \begin{itemize}
  \item there is no rook $a$ with $a_i>x_i$ and $a_j<x_j$,
  \item there is no rook $a$ with $a_i<y_i$ and $a_j>y_j$, and
  \item there is no pair of rooks $a,b$ such that
    $a_i<y_i$ and $b_i>x_i$ and $x_j<a_j<b_j<y_j$ (or, equivalently,
there is some $x_j\le j<y_j$ such that the regions
$[1,y_i-1]\times[x_j+1,j]$ and $[x_i+1,n]\times[j+1,y_j-1]$ are both
empty).
  \end{itemize}
\end{definition}

\begin{figure}
  \setlength{\unitlength}{0.45mm}
  \setlength{\fboxsep}{0.0mm}
  \begin{tabular}{c}
    (a) \\
    \begin{picture}(95,85)(-15,-30)
      \put(0,0){\makebox(10,10){$x$}}
      \put(40,10){\makebox(10,10){$y$}}
      \put(50,20){\colorbox{dark}{\makebox(30,30){}}}
      \put(10,-30){\colorbox{dark}{\makebox(30,30){}}}
      \multiput(-10,-25)(10,0){9}%
      {\begin{picture}(0,0)(0,0)\dottedline{2}(0,0)(0,70)\end{picture}}
      \multiput(-15,-20)(0,10){7}%
      {\begin{picture}(0,0)(0,0)\dottedline{2}(0,0)(90,0)\end{picture}}
    \end{picture}
  \end{tabular}
  \ \ \ \ 
  \begin{tabular}{c}
    (b) \\
    \begin{picture}(140,85)(-30,-30)
      \put(0,0){\makebox(10,10){$x$}}
      \put(70,10){\makebox(10,10){$y$}}
      \put(-30,-30){\colorbox{dark}{\makebox(30,30){}}}
      \put(80,20){\colorbox{dark}{\makebox(30,30){}}}
      \put(10,20){\colorbox{light}{\makebox(30,30){}}}
      \put(40,-30){\colorbox{light}{\makebox(30,30){}}}
      \multiput(-20,-25)(10,0){13}%
      {\begin{picture}(0,0)(0,0)\dottedline{2}(0,0)(0,70)\end{picture}}
      \multiput(-25,-20)(0,10){7}%
      {\begin{picture}(0,0)(0,0)\dottedline{2}(0,0)(130,0)\end{picture}}
    \end{picture}
  \end{tabular}
  \caption{(a) A light reduction pair. (b) A heavy reduction pair.
    The shaded areas are empty. The size of the lighter shaded areas depends on
the underlying permutation.}\label{fig:reduction_pairs}
\end{figure}
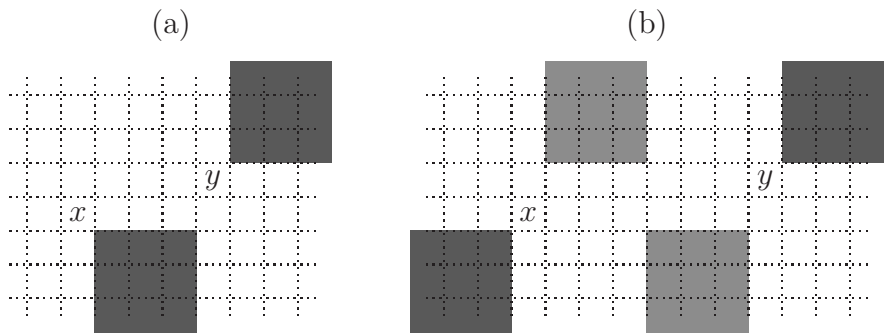

\begin{lemma}\label{lem:inductivestep}
  Let $\pi\in\hat{\fs}_n$ and assume that 
  \begin{itemize}
   \item[(a)] all $\rho\in\hat{\fs}_n$ below $\pi$
  in Bruhat order, and \item[(b)] all $\rho\in\hat{\fs}_{n-1}$ are 
chromobruhatic.
  \end{itemize}
  Then, $\pi$ is chromobruhatic
  if at least one of
  $\pi$, $\pi^{-1}$, $\pi^\crs$ and $(\pi^\crs)^{-1}$ has a reduction pair.
\end{lemma}
\begin{proof}
  If one of $\pi$, $\pi^{-1}$, $\pi^\crs$ and $(\pi^\crs)^{-1}$ has a light
  reduction pair, then by Observation~\ref{obs:trans_and_rot},
  $\pi$, $\pi^{-1}$, $\pi^\crs$ and $(\pi^\crs)^{-1}$
  all satisfy conditions (a) and (b),
  so we may assume that $\pi$ has a light reduction pair $x,y$.
  On the other hand, if none of $\pi$, $\pi^{-1}$, $\pi^\crs$
  and $(\pi^\crs)^{-1}$ has a light reduction pair, then one
  of them has a heavy reduction pair $x,y$ and we may assume that it is $\pi$.

  In either case, replace $x$ by a rook $x'$
  immediately above it, and replace $y$ by a rook $y'$
  immediately below it. The resulting permutation $\rho$ is below
  $\pi$ in the Bruhat order. Note that
  $\rho\in\hat{\fs}_n$ --- a forbidden pattern in $\rho$ must include
  both of $x'$ and $y'$ but an inspection of the forbidden patterns in
  Figure~\ref{fig:forbidden_patterns} and the reduction pair situations in
  Figure~\ref{fig:reduction_pairs} reveals that this is impossible.
  Thus, by the assumption in the lemma we conclude that $\rho$ is
  chromobruhatic.
  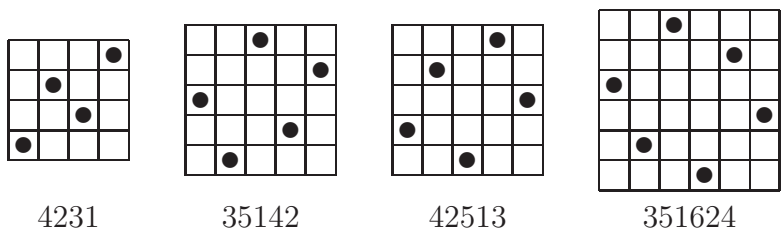
\begin{figure}
    \setlength{\unitlength}{0.4mm}
    \setlength{\fboxsep}{0.0mm}
    \begin{tabular}{cccc}
      \begin{picture}(50,60)(0,-55)
        \put(40,-10){\circle*{5}}
        \put(20,-20){\circle*{5}}
        \put(30,-30){\circle*{5}}
        \put(10,-40){\circle*{5}}
        \multiput(5,-5)(10,0){5}%
        {\begin{picture}(0,0)(0,0)\drawline(0,0)(0,-40)\end{picture}}
        \multiput(5,-5)(0,-10){5}%
        {\begin{picture}(0,0)(0,0)\drawline(0,0)(40,0)\end{picture}}
      \end{picture}
      &
      \begin{picture}(60,60)(0,-60)
        \put(30,-10){\circle*{5}}
        \put(50,-20){\circle*{5}}
        \put(10,-30){\circle*{5}}
        \put(40,-40){\circle*{5}}
        \put(20,-50){\circle*{5}}
        \multiput(5,-5)(10,0){6}%
        {\begin{picture}(0,0)(0,0)\drawline(0,0)(0,-50)\end{picture}}
        \multiput(5,-5)(0,-10){6}%
        {\begin{picture}(0,0)(0,0)\drawline(0,0)(50,0)\end{picture}}
      \end{picture}
      &
      \begin{picture}(60,60)(0,-60)
        \put(40,-10){\circle*{5}}
        \put(20,-20){\circle*{5}}
        \put(50,-30){\circle*{5}}
        \put(10,-40){\circle*{5}}
        \put(30,-50){\circle*{5}}
        \multiput(5,-5)(10,0){6}%
        {\begin{picture}(0,0)(0,0)\drawline(0,0)(0,-50)\end{picture}}
        \multiput(5,-5)(0,-10){6}%
        {\begin{picture}(0,0)(0,0)\drawline(0,0)(50,0)\end{picture}}
      \end{picture}
      &
      \begin{picture}(70,60)(0,-65)
        \put(30,-10){\circle*{5}}
        \put(50,-20){\circle*{5}}
        \put(10,-30){\circle*{5}}
        \put(60,-40){\circle*{5}}
        \put(20,-50){\circle*{5}}
        \put(40,-60){\circle*{5}}
        \multiput(5,-5)(10,0){7}%
        {\begin{picture}(0,0)(0,0)\drawline(0,0)(0,-60)\end{picture}}
        \multiput(5,-5)(0,-10){7}%
        {\begin{picture}(0,0)(0,0)\drawline(0,0)(60,0)\end{picture}}
      \end{picture}
      \\
      4231 & 35142 & 42513 & 351624
    \end{tabular}
    \caption{The four forbidden patterns.}
    \label{fig:forbidden_patterns}
  \end{figure}

  {\bf Case 1: $x,y$ is a light reduction pair in $\pi$.}
  What permutations
  are below $\pi$ but not below $\rho$ in the Bruhat order? Note that
  $\rho$ has the same bubbles as $\pi$, plus an additional bubble
  immediately above
  $y'$, i.e.~at the position of $y$. Now, Lemma~\ref{lem:bubbles} yields
  that the only permutations 
  below $\pi$ that are not below $\rho$ are the ones with a rook at the
  position of $y$.
  These are in one-one correspondence with the permutations weakly below
  the permutation $\pi-y\in\fs_{n-1}$ that we obtain by deleting $y$ from $\pi$
  together with its row and column.
  Thus, we have
  \begin{equation}\label{eq:light_br_recurrence}
    \br(\pi)=\br(\rho)+\br(\pi-y).
  \end{equation}
  Now consider the inversion graphs of $\pi$, $\rho$ and $\pi-y$.
  It is not hard to show that
  $G_\rho$ is isomorphic to the graph $G_\pi\setminus\{x,y\}$
  obtained by deletion of the edge $\{x,y\}$. Since all
  neighbors of $y'$ are also neighbors of $x'$ in $G_\rho$,
  the graph $G_{\pi-y}=G_{\rho-y'}$ is isomorphic to
  the graph $G_\pi/\{x,y\}$ obtained by contraction of the edge $\{x,y\}$.
  It is a well-known fact that,
  for any edge $e$ in any simple graph $G$, the number of acyclic
  orientations satisfies the recurrence relation
  $\ao(G)=\ao(G\setminus e)+\ao(G/e)$.
  Thus, in our case we get
  \begin{equation}\label{eq:light_ao_recurrence}
    \ao(\pi)=\ao(\rho)+\ao(\pi-y).
  \end{equation}
  The right-hand sides of equations~(\ref{eq:light_br_recurrence})
  and~(\ref{eq:light_ao_recurrence}) are equal since $\rho$ and
  $\pi-y$ are chromobruhatic. We conclude that $\br(\pi)=\ao(\pi)$
  so that $\pi$ also is chromobruhatic.

  {\bf Case 2: $x,y$ is a heavy reduction pair in $\pi$, and none of 
    $\pi$, $\pi^{-1}$, $\pi^\crs$ and $(\pi^\crs)^{-1}$ has a light
    reduction pair.}
  Since $y,x$ is not a light reduction pair in $\pi^\crs$,
  there exists a rook $a$ in the region $A=[1,y_i-1]\times[x_j+1,y_j-1]$.
  Analogously, since $x,y$ is not a light reduction pair in $\pi$,
  there exists a rook $b$ in the region $B=[x_i+1,n]\times[x_j+1,y_j-1]$.
  As can be seen in Figure~\ref{fig:heavy_swap}, the right hulls of
  $\pi$ and $\rho$ are the same except for the two squares containing 
$x$ and $y$,
  which belong to $H_R(\pi)$ but not to $H_R(\rho)$.
  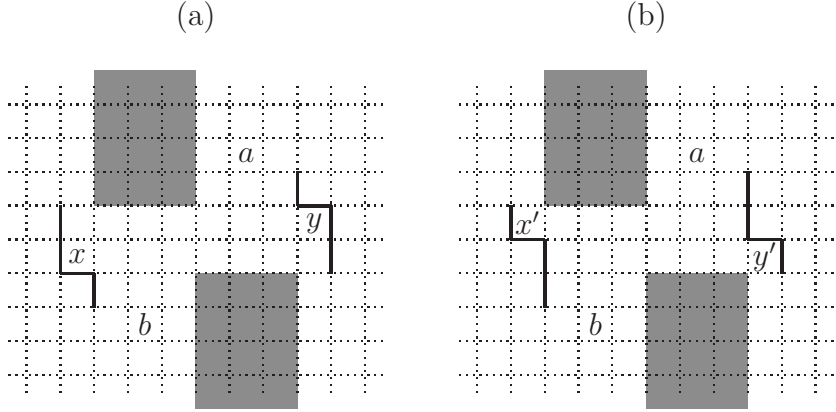
\begin{figure}
    \setlength{\unitlength}{0.45mm}
    \setlength{\fboxsep}{0.0mm}
    \begin{tabular}{c}
      (a) \\
      \begin{picture}(110,110)(-15,-40)
        \put(0,0){\makebox(10,10){$x$}}
        \put(70,10){\makebox(10,10){$y$}}
        \put(20,-20){\makebox(10,10){$b$}}
        \put(50,30){\makebox(10,10){$a$}}
        \put(10,20){\colorbox{light}{\makebox(30,40){}}}
        \put(40,-40){\colorbox{light}{\makebox(30,40){}}}
        {
          \thicklines
          \drawline(0,20)(0,0)(10,0)(10,-10)
          \drawline(70,30)(70,20)(80,20)(80,0)
        }
        \multiput(-10,-35)(10,0){11}%
        {\begin{picture}(0,0)(0,0)\dottedline{2}(0,0)(0,90)\end{picture}}
        \multiput(-15,-30)(0,10){9}%
        {\begin{picture}(0,0)(0,0)\dottedline{2}(0,0)(110,0)\end{picture}}
      \end{picture}
    \end{tabular}
    \ \ \ \ 
    \begin{tabular}{c}
      (b) \\
      \begin{picture}(110,110)(-15,-40)
        \put(0,10){\makebox(10,10){$x'$}}
        \put(70,0){\makebox(10,10){$y'$}}
        \put(20,-20){\makebox(10,10){$b$}}
        \put(50,30){\makebox(10,10){$a$}}
        \put(10,20){\colorbox{light}{\makebox(30,40){}}}
        \put(40,-40){\colorbox{light}{\makebox(30,40){}}}
        {
          \thicklines
          \drawline(0,20)(0,10)(10,10)(10,-10)
          \drawline(70,30)(70,10)(80,10)(80,0)
        }
        \multiput(-10,-35)(10,0){11}%
        {\begin{picture}(0,0)(0,0)\dottedline{2}(0,0)(0,90)\end{picture}}
        \multiput(-15,-30)(0,10){9}%
        {\begin{picture}(0,0)(0,0)\dottedline{2}(0,0)(110,0)\end{picture}}
      \end{picture}
    \end{tabular}
    \caption{(a) The heavy reduction pair $x,y$ in $\pi$.
      The shaded areas are empty and the thick lines show segments of
      the border of the right hull of $\pi$. (b) The right hull of $\rho$
      is the same as that of $\pi$, except for the two squares of $x$ and $y$.}
    \label{fig:heavy_swap}
  \end{figure}

  By Lemma~\ref{lem:right_hull} and inclusion-exclusion, we get
  \begin{equation}\label{eq:heavy_br_recurrence} 
    \br(\pi)=\br(\rho)+\br(\pi-x)+\br(\pi-y)-\br(\pi-x-y)
  \end{equation}
  where $\pi-x-y \in \fs_{n-2}$ is the permutation whose rook diagram is 
obtained by
  deleting both of $x$ and $y$ together with their rows and columns.

  Now, for any permutation $\sigma$,
  let $\chi_\sigma(t)=\chi_{G_\sigma}(t)$
  denote the chromatic polynomial of the inversion graph
  $G_\sigma$ (so for each positive integer $n$, $\chi_{G_\sigma}(n)$
  is the number of vertex colorings with at most $n$ colors
  such that neighboring vertices get distinct colors.
  The following argument is based on an idea by Postnikov.
  It is a well-known fact that $\ao(G)=(-1)^n\chi_G(-1)$ for any graph $G$
  with $n$ vertices.
  Since $G_\rho=G_\pi\setminus\{x,y\}$, the difference
  $\chi_\rho(t)-\chi_\pi(t)$ is the number of $t$-colorings of
  $G_\rho$ where $x'$ and $y'$ have the same color.
  
  Let $\mathcal{C}$
  be any $t$-coloring of $G_{\pi-x-y}$ using, say, $\alpha$ different
  colors for the vertices in $A$ and $\beta$ different colors for those
  in $B$. Since the subgraph of $G_\pi$ induced by $A\cup B$
  is a complete bipartite graph, the coloring $\mathcal C$ must use
  $\alpha+\beta$ different colors for the vertices in $A\cup B$.
  We can extend $\mathcal C$ to a coloring of $G_{\pi-y}$ by
  coloring the vertex $x$ with any of the $t-\alpha$ colors that
  are not used for the vertices in $A$. Analogously, we
  can extend $\mathcal C$ to a coloring of $G_{\pi-x}$ by
  coloring the vertex $y$ with any of the $t-\beta$ colors that
  are not used in $B$. Finally, we can extend $\mathcal C$
  to a coloring of $G_\rho$ where $x'$ and $y'$ have the same color,
  by choosing this color among the $t-\alpha-\beta$ colors that
  are not used for the vertices in $A\cup B$.
  Summing over all $t$-colorings $\mathcal C$ of $G_{\pi-x-y}$
  yields
  \begin{align*}
    \chi_\rho(t)-\chi_\pi(t)&=\sum_{\mathcal C}(t-\alpha-\beta) \\
    &=\sum_{\mathcal C}(t-\beta)+\sum_{\mathcal C}(t-\alpha)-
\sum_{\mathcal C}t \\
    &=\chi_{\pi-x}(t)+\chi_{\pi-y}(t)-t\chi_{\pi-x-y}(t).
  \end{align*}
  Using that $\ao(G)=(-1)^n\chi_G(-1)$ for a graph $G$ with $n$ vertices,
  we finally obtain
  \begin{equation}\label{eq:heavy_ao_recurrence}
    \ao(\pi)=\ao(\rho)+\ao(\pi-x)+\ao(\pi-y)-\ao(\pi-x-y).
  \end{equation}
  The right-hand sides of equations~\ref{eq:heavy_br_recurrence}
  and~\ref{eq:heavy_ao_recurrence} are equal by the assumption in the lemma.
  Thus, $\br(\pi)=\ao(\pi)$ and we conclude that $\pi$ is
  chromobruhatic.
\end{proof}

Let $\pi\in\fs_n$ be any nonidentity permutation.
Then there is a pair of rooks $x,y$ that is the first descent of
$\pi$, i.e.~$x_i=\min\{i\,:\,i\pi<(i-1)\pi\}$ and $y_i=x_i-1$.
Analogously, let $\bar x,\bar y$ be the first descent of $\pi^{-1}$,
i.e.~$\bar x_j=\min\{j\,:\,j\pi^{-1}<(j-1)\pi^{-1}\}$ and
$\bar y_j=\bar x_j-1$.
\begin{proposition}\label{pr:reduction}
  For any nonidentity $\pi\in\hat{\fs}_n$,
  either $x,y$ is a reduction pair in $\pi$ or
  $\bar x,\bar y$ is a reduction pair in $\pi^{-1}$, or both.
\end{proposition}
\begin{proof}
  We suppose neither of $x,y$ and $\bar x,\bar y$ is a reduction pair,
  and our goal is to find a forbidden pattern.

  If $\pi(1)=1$ it suffices to look at the rook configuration on
  the smaller board $[2,n]\times[2,n]$ since the pairs $x,y$ and 
  $\bar x,\bar y$
  on that board are not reduction pairs either.
  Thus, we may assume that $\pi(1)>1$.

  Let $z$ be the rook in row 1 and let $\bar z$ be the rook in column 1.
  From our assumption that $x,y$ is not a light reduction pair in $\pi$, and

  the fact that the rooks $x,y$ represent the first descent in $\pi$, 
  it follows that there is a rook $a$ in the
  region $A=[x_i+1,n]\times[x_j+1,y_j-1]$. If $x$ is in column 1, our
  assumption that $x,y$ is not a heavy reduction pair implies that
  $y\ne z$ and that there is a rook $b$ in the region
  $B=[x_i+1,n]\times[z_j+1,y_j-1]$,
  because $z$ is the leftmost rook in the rows above $y$. Analogously,
  since $\bar x,\bar y$ is not a reduction pair in $\pi^{-1}$,
  there is a rook $\bar a\in\bar 
  A=[\bar x_i+1,\bar y_i-1]\times[\bar x_j+1,n]$,
  and if $\bar x$ is in the first row, then $\bar y\ne\bar z$ and
  there is a rook
  $\bar b\in\bar B=[\bar z_i+1,\bar y_i-1]\times[\bar x_j+1,n]$.

  By the construction of $x$,
  all rooks in rows above of $x$ are weakly to the right of $z$,
  so $\bar z$ is weakly below $x$. Analogously, $z$ is weakly to the
  right of $\bar x$.
  This implies that either $\bar x$ is weakly below $x$, or $\bar x=z$,
  and analogously, either $x$ is weakly to the right of $\bar x$, or
  $x=\bar z$.

  {\bf Case 1: $x\ne\bar z$ and $\bar x\ne z$ as in Figure~\ref{fig:case1}.}
  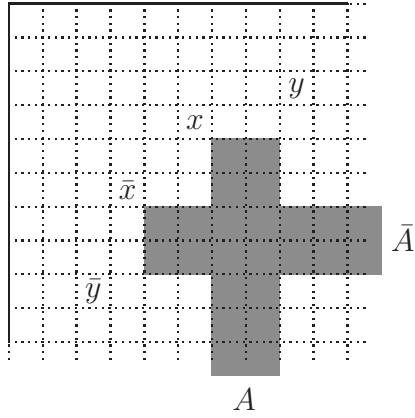
\begin{figure}
    \setlength{\unitlength}{0.45mm}
    \setlength{\fboxsep}{0.0mm}
    \begin{picture}(120,120)(10,-120)
      \put(40,-60){\makebox(10,10){$\bar x$}}
      \put(30,-90){\makebox(10,10){$\bar y$}}
      \put(60,-40){\makebox(10,10){$x$}}
      \put(90,-30){\makebox(10,10){$y$}}
      \put(50,-80){\colorbox{light}{\makebox(70,20){}}}
      \put(70,-110){\colorbox{light}{\makebox(20,70){}}}
      \put(70,-80){\colorbox{light}{\makebox(20,20){}}}
      \put(123,-73){$\bar A$}
      \put(76,-120){$A$}
      \drawline(10,-100)(10,0)(110,0)
      \multiput(10,0)(10,0){11}%
      {\begin{picture}(0,0)(0,0)\dottedline{2}(0,0)(0,-105)\end{picture}}
      \multiput(10,0)(0,-10){11}%
      {\begin{picture}(0,0)(0,0)\dottedline{2}(0,0)(105,0)\end{picture}}
    \end{picture}
    \caption{The situation of case 1. It is possible that $x=\bar x$.}
    \label{fig:case1}
  \end{figure}
  If $a$ is above $\bar y$, then the rooks $y,x,a,\bar y$
  form the forbidden pattern 4231.
  Analogously, if $\bar a$ is to the left of $y$,
  then the rooks $y,\bar x,\bar a,\bar y$
  form the forbidden pattern 4231. Finally, if $a$ is below $\bar y$
  and $\bar a$ is to the right of $y$, then
  the rooks $y,x,\bar a,\bar y,a$ form the forbidden pattern 42513.

  {\bf Case 2: $x=\bar z$ but $\bar x\ne z$ as in Figure~\ref{fig:case2}.}
  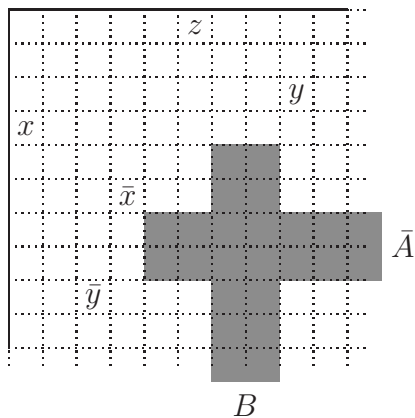
\begin{figure}
    \setlength{\unitlength}{0.45mm}
    \setlength{\fboxsep}{0.0mm}
    \begin{picture}(120,120)(10,-120)
      \put(40,-60){\makebox(10,10){$\bar x$}}
      \put(30,-90){\makebox(10,10){$\bar y$}}
      \put(10,-40){\makebox(10,10){$x$}}
      \put(60,-10){\makebox(10,10){$z$}}
      \put(90,-30){\makebox(10,10){$y$}}
      \put(50,-80){\colorbox{light}{\makebox(70,20){}}}
      \put(70,-110){\colorbox{light}{\makebox(20,70){}}}
      \put(70,-80){\colorbox{light}{\makebox(20,20){}}}
      \put(123,-73){$\bar A$}
      \put(76,-120){$B$}
      \drawline(10,-100)(10,0)(110,0)
      \multiput(10,0)(10,0){11}%
      {\begin{picture}(0,0)(0,0)\dottedline{2}(0,0)(0,-105)\end{picture}}
      \multiput(10,0)(0,-10){11}%
      {\begin{picture}(0,0)(0,0)\dottedline{2}(0,0)(105,0)\end{picture}}
    \end{picture}
    \caption{The situation of case 2. (If we transpose the diagram and
      let each letter change places with its barred variant, we obtain the
      situation of case 3.)}
    \label{fig:case2}
  \end{figure}
  As before, if $\bar a$ is to the left of $y$,
  then the rooks $y,\bar x,\bar a,\bar y$
  form the forbidden pattern 4231.
  If $b$ is above $\bar y$, then
  $z,y,x,b,\bar y$ form the pattern 35142. Finally, if
  $\bar a$ is to the right of $y$ and $b$ is below
  $\bar y$, then $y,\bar x,\bar a,\bar y,b$ form the pattern 42513.

  {\bf Case 3: $x\ne\bar z$ but $\bar x= z$.}
  This is just the ``transpose'' of case 2.

  {\bf Case 4: $x=\bar z$ and $\bar x=z$ as in Figure~\ref{fig:case4}.}
  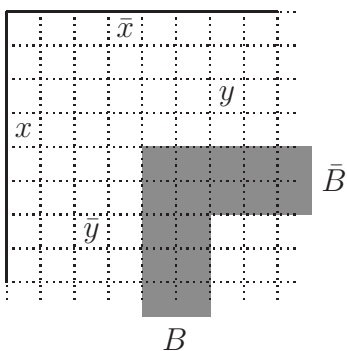
\begin{figure}
    \setlength{\unitlength}{0.45mm}
    \setlength{\fboxsep}{0.0mm}
    \begin{picture}(100,100)(10,-100)
      \put(30,-70){\makebox(10,10){$\bar y$}}
      \put(10,-40){\makebox(10,10){$x$}}
      \put(40,-10){\makebox(10,10){$\bar x$}}
      \put(70,-30){\makebox(10,10){$y$}}
      \put(50,-60){\colorbox{light}{\makebox(50,20){}}}
      \put(50,-90){\colorbox{light}{\makebox(20,50){}}}
      \put(50,-60){\colorbox{light}{\makebox(20,20){}}}
      \put(103,-53){$\bar B$}
      \put(56,-100){$B$}
      \drawline(10,-80)(10,0)(90,0)
      \multiput(10,0)(10,0){9}%
      {\begin{picture}(0,0)(0,0)\dottedline{2}(0,0)(0,-85)\end{picture}}
      \multiput(10,0)(0,-10){9}%
      {\begin{picture}(0,0)(0,0)\dottedline{2}(0,0)(85,0)\end{picture}}
    \end{picture}
    \caption{The situation of case 4.}
    \label{fig:case4}
  \end{figure}
  If there is a rook $\tilde b\in B\cap\bar B$, then
  $\bar x,y,x,\tilde b,\bar y$ form the pattern 35142.
  But if $b$ is below $\bar y$ and $\bar b$ is to the right
  of $y$, then $z,y,\bar z,\bar b,\bar y,b$ form the last forbidden
  pattern 351624.
\end{proof}
Combining Lemma~\ref{lem:inductivestep} and
Proposition~\ref{pr:reduction}, yields the
following two corollaries via induction.
\begin{corollary}\label{co:surjective}
  A permutation is chromobruhatic if it avoids
the patterns $4231$,
$35142$, $42513$ and $351624$.
\end{corollary}
Recall that the right and left weak orders on $\fs_n$ are defined
by $u\le_{R}w\Leftrightarrow\INV(u)\subseteq\INV(w)$ and
$u\le_{L}w\Leftrightarrow\INV(u^{-1})\subseteq\INV(w^{-1})$.
The \emph{two-sided weak order} is the transitive closure
of the union of the right and left weak orders.
\begin{corollary}
  Every chromobruhatic permutation is connected to the identity permutation
  via a saturated chain of chromobruhatic permutations in the two-sided 
  weak order.
\end{corollary}

\section{Another characterization of permutations that avoid the four patterns}
\label{sec:axel}
In this section we demonstrate a feature of the injection $\phi:\CL
\to [e,w]$ that we call the ``going-down property''. As a
consequence, yet another characterization of permutations that avoid $4231$,
$35142$, $42513$ and $351624$ is deduced. It 
implies, in particular, that avoidance of these patterns is a
combinatorial property of the principal ideal a permutation generates
in the Bruhat order.

\begin{lemma}\label{le:down}
Let $(W,S)$ be any finitely generated Coxeter system. Suppose
$s_1\dotsm s_k$ is a reduced expression for $w\in W$. Define $t_i =
s_1\dotsm s_i\dotsm s_1\in T_w$.
Assume there exist $1\leq i_1<\dotsb <
i_m\leq k$ such that $t_{i_1}\dotsm t_{i_m}w = u$ and that the string
$(i_m, \dotsc, i_1)$ is lexicographically maximal with this
property (for fixed $m$ and $u$).
Then, $w>t_{i_m}w>t_{i_{m-1}}t_{i_m}w > \dotsb > t_{i_1}\dotsm
t_{i_m}w = u$. 
\begin{proof}
In order to arrive at a contradiction, let us assume $t_{i_j}\dotsm t_{i_m}w >
t_{i_{j+1}}\dotsm t_{i_m}w = b$. The strong exchange property
(Proposition \ref{pr:SEP}) implies that
an expression for $b$ can be obtained from $s_1\dotsm
\wh{s_{i_j}}\dotsm \wh{s_{i_m}}\dotsm s_k$ by deleting a letter
  $s_x$.

If $x<i_j$, then $t_{i_j}=t_x$ and $w = t_{i_j}^2w = s_1\dotsm
\wh{s_x}\dotsm \wh{s_{i_j}} \dotsm s_k$, contradicting the fact our 
original expression for $w$ is reduced.

Now suppose $x> i_j$; say $i_j\leq i_l<x<i_{l+1}$ (where we have
defined $i_{m+1}=k+1$). Hence, $u = t_{i_1}\dotsm
t_{i_{j-1}}t_{i_{j+1}}\dotsm t_{i_l}t_xt_{i_{l+1}}\dotsm
t_{i_m}w$. This, however, contradicts the maximality of $(i_1, \dotsc, i_m)$.
\end{proof}
\end{lemma}

\begin{proposition}[Going-down property of $\phi$]\label{pr:down}
Choose $C = \{\Hz = X_0 \lhd X_1 \lhd \dotsb \lhd X_m\}\in \CL$. Assume
$\lambda (X_{i-1}\lhd X_i) = H_{j_i}$ with corresponding reflection $t_{j_i}$. Then, $t_{j_i}\dotsm t_{j_m}w <
t_{j_{i+1}}\dotsm t_{j_m}w$ for all $i$.
\begin{proof}
Applying Lemma \ref{le:down}, it suffices to show that $(j_m,
\dotsc, j_1)$ is lexicographically maximal in the set $\{(p_m, \dotsc,
p_1)\in [k]^m \mid p_m>\dotsb > p_1\text{ and }t_{p_1}\dotsm
t_{p_m}=t_{j_1}\dots t_{j_m}\}$. Let us deduce a contradiction by assuming that $(j_m^\prime, \dotsc, j_1^\prime)$ is a
lexicographically larger sequence in this set. Suppose $i$
is the largest index for which $j_i\neq j_i^\prime$. We have
$H_{j_i^\prime} \supseteq X_i$, because,
by Proposition~\ref{pr:Carter} and Lemma~\ref{le:dimension},
$X_i$ is the fixed point space
of $t_{j_1}\dotsm t_{j_i} = t_{j_1^\prime}\dotsm t_{j_i^\prime}$ which
is an element of absolute length $i$. Observing that $j_i^\prime >
j_i$, i.e.\ $H_{j_i}>H_{j_i^\prime}$, the construction of $\lambda$
implies $\lambda(X_{\alpha-1}\lhd X_\alpha) \leq
H_{j^\prime_i}$ for some $\alpha\in [i]$. However, this contradicts
the fact that $\lambda(X_{\alpha-1}\lhd X_\alpha) \geq H_{j_i}$ for
all such $\alpha$. 
\end{proof}
\end{proposition}

Given $u\leq w \in W$, let $a\ell(u,w)$ denote the directed distance
from $u$ to $w$ in the directed graph (the {\em Bruhat graph}
\cite{dyer}) on $W$ whose edges are given by $x\to tx$ whenever $t\in
T$ and $\ell(x) < \ell(tx)$. Observe that $a\ell(u,w) \geq
\aell(uw^{-1})$ in general.

\begin{theorem} \label{th:characterization}
Let $w\in \Sn$. The following assertions are equivalent:
\begin{itemize}
\item $w$ avoids $4231$, $35142$, $42513$ and $351624$.
\item $\aell(uw^{-1}) = a\ell(u,w)$ for all $u < w$. 
\end{itemize}
\begin{proof}
If $w$ avoids the given patterns, $\phi$ is surjective. Proposition
\ref{pr:down} then shows that for any $u<w$ there is a directed path
of length $\aell(uw^{-1})$ from $u$ to $w$ in the Bruhat graph.

For the converse implication, suppose $w$ contains at least one of the
patterns. By 
the proof of Theorem \ref{th:Avoiding}, there exists some $u<w$
such that $uw^{-1}$ 
cannot be written as a product of $\aell(uw^{-1})$ inversions of
$w$. On the other hand, whenever there is a directed path from $u$ to
$w$ of length 
$p$, then $uw^{-1}$ can be written as a product of $p$ inversions of
$w$ (this follows from the strong exchange property). Hence,
$a\ell(u,w) > \aell(uw^{-1})$. 
\end{proof}
\end{theorem}

\begin{corollary}
Suppose $w_1\in \Sn$ avoids $4231$, $35142$, $42513$ and $351624$
whereas $w_2\in \Sn$ does not. Then, $[e,w_1]\not \cong [e,w_2]$ as posets.
\begin{proof}
Let $u\leq w \in W$. Denote by $\BG(u,w)$ the subgraph of the
Bruhat graph on $W$ induced by the elements in the Bruhat interval $[u,w]$. 
It is known \cite[Proposition 3.3]{dyer} that the isomorphism type
of $[u,w]$ determines the isomorphism type of $\BG(u,w)$.

Now suppose $w\in \Sn$ contains one of the four patterns. In the
proof of Theorem 
\ref{th:Avoiding}, we produced elements $u < w$ such
that $uw^{-1}$ cannot be written as a product of $\aell(uw^{-1})$
inversions of $w$. A closer examination of these elements reveals
that, for each such $u$, there is a transposition $t$ such that $tu<u$
and $\aell(tuw^{-1})=a\ell(tu,w) = \aell(uw^{-1})-1$.\footnote{For example, if
  the pattern $42513$ occurs in positions $n_1, \dotsc, n_5$, we have
  $uw^{-1}=(n_2\, n_5\, n_3)(n_1\, n_4)$. Observe that $n_2u = n_5w$ and
  $n_3u = n_2w$. Hence, $t=(n_2\, n_3)$ with $(n_2,n_3) \in \INV(u)$. Now, $tuw^{-1} =
  (n_3\, n_5)(n_1\, n_4)$, $\aell(tuw^{-1})=2$ and $tu \to (n_3\, n_5)tu
  \to (n_1\, n_4)(n_3\, n_5)tu = w$ is a directed path in $\BG(e,w)$ of
  length $2$. The remaining three cases are similar.} Thus, $\BG(e,w)$
contains an undirected path from $u$ to $w$ of length
$\aell(uw^{-1})$. Therefore, it is possible to
determine from the combinatorial type of $[e,w]$ that it
contains an element $u$ with $a\ell(u,w)>\aell(uw^{-1})$. 
\end{proof} 
\end{corollary}

\section{Inequality of Betti numbers} \label{S:Betti}

In this section we use the bijection $\phi$ to derive, for $w\in \hat{\fs}_n$, 
inequalities relating the ranks of the cohomology groups of the
complexified hyperplane arrangement $\rcca_w$ and the closure of the
cell corresponding to $w$ in the Bruhat decomposition of the flag manifold.

Let $B$ be a Borel subgroup of $G=GL_n({\mathbb C})$.  The Schubert cells
(or Bruhat cells) ${BwB/B}$ ($w \in S_n$) determine a cell decomposition of the 
complex flag manifold $G/B$.  The closure of each such cell
admits a regular decomposition into cells indexed by permutations in
the Bruhat interval $[e,w]$, that is, $\overline{BwB/B}=
\cup_{\pi\le w} B\pi B/B$. All Schubert cells are even-dimensional.  It follows that 
$\sum_{i}\beta^{2i}(\overline{BwB/B})q^i=\sum_{\pi\le w} q^{\ell(\pi)}$.
That is, $\beta^{2i}(\overline{BwB/B})$ counts the number of elements
$u\in [e,w]$ with $\ell(u)=i$. This is well known, see for instance 
\cite{borel, fulton, GR}. 

The linear equations determining the hyperplanes in an arrangement $\ca$ in $\rr^n$ also define hyperplanes in $\mathbb{C}^n$. These complex hyperplanes 
  yield the {\em  complexified} arrangement $\cca$.

For the complexified hyperplane arrangement $\rcca_w$ we have the Orlik-Solomon formula for the Betti numbers of the complement of
a complex hyperplane arrangement,
\[\beta^{i}({\mathbb C}^{n} \setminus \cup \rcca_w)=
\sum_{x\in L_w: \text{rank}(x)=i}|\mu(\hat 0,x)|.\]
See \cite{Bj92} for background on subspace 
arrangements. As noted above, the theory for 
lexicographic shellability of posets \cite{bjorner} says that 
$|\mu(\hat 0,x)|$ is the number of descending saturated 
chains in the EL-labeling $\lambda$ 
starting at $\hat 0$ and ending at $x$.

\begin{proposition}\label{pr:betprop}
  For any permutation $w\in \Sn$ that avoids the patterns
  4231, 35142, 42513, and 351624, we have for $r\ge 0$ that
  \begin{enumerate}
  \item{}$
    \sum_{i=0}^r\beta^{2(\ell(w)-i)}(\overline{BwB/B})\le
    \sum_{i=0}^r \beta^{i}({\mathbb C}^n \setminus \rcca_w),$
  \item{}$\sum_{j=0}^r\beta^{2(\ell(w)-2j)}(\overline{BwB/B})\le
    \sum_{j=0}^r \beta^{2j}({\mathbb C}^{n} \setminus \rcca_w)$ and
  \item{}$\sum_{j=0}^r\beta^{2(\ell(w)-2j-1)}(\overline{BwB/B})\le
    \sum_{j=0}^r \beta^{2j+1}({\mathbb C}^{n} \setminus \rcca_w)$.
  \end{enumerate}
  When $r$ is maximal, that is, when the sum is taken over all non-zero Betti numbers,
  we have equality. This occurs when $r=\ell(w)$, $r=\lfloor\ell(w)/2\rfloor$ 
  and $r=\lfloor(\ell(w)-1)/2\rfloor$, respectively.
\end{proposition}

\begin{proof}
  We use the notation introduced in Section \ref{sec:inj}.  Let $s_1\ldots s_k$ be a reduced expression for $w$.  The right hand side in (1) counts chains $C\in\CL$ 
  of length at most $r$. Each such chain of length $i$
  gives a word $p(C)$ of length $i$ in the alphabet $t_1,\ldots,t_k$. By Lemma \ref {le:dimension}
  we have $\aell(p(C))=i$ and thus $\ell(\phi(C))\le\ell(w)-i$. By Theorem
  \ref{th:injective} $\phi$ is injective and the inequality follows.

  Since multiplication by a transposition $t_j$ always changes the 
  length of $w \in S_n$ by an odd number, the other two inequalities follow.

  The map $\phi$ is by Theorem \ref{th:main} a bijection between chains with 
  descending labels and elements in the Bruhat interval $[e,w]$ which gives
  equality of the number of plausible words in the $t_j$s and $s_j$s
  respectively.
\end{proof}
                               
Note that these inequalities are not true in general for permutations
not avoiding the four patterns.
In fact, if $w\notin \hat{\fs}_n$ and $r=\ell(w)$ we know by
Theorem~\ref{th:Avoiding}
that the inequality (1) does not hold.

\section{Chromatic polynomials and smooth permutations}
\label{sec:chromo}
Recall the directed distance $a\ell(u,w)$ defined prior to Theorem \ref{th:characterization}. In this section we will use the injective map $\phi$ from
Proposition~\ref{pr:map} to show that the chromatic polynomial
$\chi_{G_w}(t)$ of the inversion graph $G_w$ of $w\in\hat{\fs}_n$
keeps track of the transposition distance $a\ell(u,w)$ of
elements $u\in[e,w]$. We follow
Postnikov and sometimes call a permutation \emph{chromobruhatic}
if it avoids the four forbidden patterns.
\begin{theorem}\label{th:chromatic}
For any permutation $w\in\fs_n$, the polynomial
identity
\[
\sum_{u\in[e,w]}q^{a\ell(u,w)}=(-q)^n \chi_{G_w}(-q^{-1}),
\]
holds if and only if $w$ avoids the patterns
4231, 35142, 42513 and 351624.
\end{theorem}
\begin{proof}
It is well-known (see e.g.~\cite{stanley}) that
\[
\chi_{G_w}(t)=\sum_{X\in L_w}\mu(X)t^{\dim X}=
\sum_{X\in L_w}(-1)^{\codim X}|\mu(X)|t^{\dim X}.
\]
Lemma~\ref{le:dimension} implies that
if $u=\phi(\Hz=X_0\lhd X_1\lhd\dotsb\lhd X_m)$ then
$\aell(uw^{-1})=m=\codim(X_m)$. 
If $w$ avoids the four patterns we have by Theorem \ref{th:characterization} 
that $\aell(uw^{-1})=a\ell(u,w)$ and thus
\[
\sum_{X\in L_w}(-1)^{\codim X}|\mu(X)|t^{\dim X}
=\sum_{u\in[e,w]}(-1)^{a\ell(u,w)}t^{n-a\ell(u,w)},
\]
since $\phi$ is bijective. If $w$ does contain one of the four patterns 
Theorem \ref{th:Avoiding} gives inequality by substituting $t=-1$. 

Finally, make the substitution $t=-q^{-1}$.
\end{proof}

A well-known criterion, due to Lakshmibai and Sandhya~\cite{lakshmibaisandhya},
says that for a permutation $w\in\fs_n$, the Schubert variety
$\overline{BwB/B}$ is smooth if and only if $w$ avoids the patterns 3412 and 4231.
Let us say that such a permutation itself is \emph{smooth}.  Note that every smooth permutation is chromobruhatic.

Given $w\in\fs_n$ and regions $r$ and $r'$ of $\rr^{n-1} \setminus \rca_w$, let $d(r,r')$
denote the number of hyperplanes of $\rca_w$ that separate $r$ and $r'$.
Let $r_0$ be the region that contains the point $(1,\dotsc,n)$,
and define $R_w(q)=\sum_r q^{d(r_0,r)}$, where the sum is taken over
all regions of $\rca_w$.

Recently, Oh, Postnikov, and Yoo~\cite{ohpostnikovyoo}
showed that the Poincar\'e polynomial
$\sum_{u\in[e,w]}q^{\ell(u)}$ equals $R_w(q)$ if and only
if $w$ is smooth. They also link this polynomial to
the chromatic polynomial $\chi_{G_w}(t)$, and they are able
to compute the latter, which is very useful for us.

An index $r\in\{1,\dotsc,n\}$ is a \emph{record position} 
of a permutation $w\in\fs_n$ if $rw>\max\{1w,\dotsc,(r-1)w\}$.
For $i=1,\dotsc,n$, let $r_i$ and $r'_i$ be the record positions
of $w$ such that $r_i\le i<r'_i$ and there are no other record positions
between $r_i$ and $r'_i$. (Set $r'_i=+\infty$ if there
are no record positions greater than $i$.) Let
\[
e_i=\#\{j\,|\,r_i\le j<i,\ jw>iw\}
+\#\{k\,|\,r'_i\le k\le n,\ kw<iw\}.
\]
\begin{theorem}[Oh, Postnikov, Yoo]\label{th:smoothchromo}
For any smooth permutation $w\in\fs_n$, the chromatic polynomial
of the inversion graph of $w$ is given by
$\chi_{G_w}(t)=(t-e_1)(t-e_2)\dotsm(t-e_n)$.
\end{theorem}
Combining this with Theorem~\ref{th:chromatic} allows us to
compute the transposition distance generating function
$\sum_{u\in[e,w]}q^{a\ell(u,w)}$ for any smooth permutation
$w\in\fs_n$.

\section{Example: the permutation $w=4132$}\label{sec:example}
{
Consider the symmetric group $W=\fs_4$ generated by the
adjacent transpositions
$S=\{s_1=(1\, 2),\ s_2=(2\, 3),\ s_3=(3\, 4)\}$, and let
$w=4132=s_1s_2s_3s_2$ so that
$t_1=s_1=(1\, 2)$, $t_2=s_1s_2s_1=(1\, 3)$, $t_3=s_1s_2s_3s_2s_1=(1\, 4)$,
and $t_4=s_3=(3\, 4)$.
The intersection lattice $L_W$ is isomorphic to the
lattice of partitions of the set $\{1,2,3,4\}$ ordered by
refinement. (For instance, the partition $13|24$ corresponds to
the set
$\{(x_1,x_2,x_3,x_4)\in\rr^4\,|\,x_1=x_3\ \text{and}\ x_2=x_4\}\in L_W$.)
With this notation, the lattice $L_w$ looks like this:
\begin{center}
$\xymatrix @=0.8cm @C=0.4cm{
& & & 1234 & & & \\
123|4 \ar@{-}[urrr]|4 & & 124|3 \ar@{-}[ur]|4 & &
12|34  \ar@{-}[ul]|3 & & 134|2 \ar@{-}[ulll]|1 \\
12|3|4 \ar@{-}[u]|2 \ar@{-}[urr]|(.65)3 \ar@{-}[urrrr]|4 & &
13|2|4 \ar@{-}[ull]|(.65)1 \ar@{-}[urrrr]|4 & &
14|2|3 \ar@{-}[ull]|1 \ar@{-}[urr]|(.35)4 & &
1|2|34 \ar@{-}[ull]|(.35)1 \ar@{-}[u]|3 \\
& & & 1|2|3|4 \ar@{-}[ulll]|1 \ar@{-}[ul]|2 \ar@{-}[ur]|3 \ar@{-}[urrr]|4
& & &
}$
\end{center}
Here the coverings are labelled by indices; for instance,
since $\lambda(12|3|4\lhd 12|34)=H_4$, that edge is labelled
by 4. After finding the decreasing chains $C\in \CL$,
we obtain the following table.
\[
\begin{array}{l|l|lcc}
C & p(C) & p(C)w & & \\
\hline
\Hz & e & s_1s_2s_3s_2 & = & 4132 \\
\Hz\lhd 12|3|4 & t_1 & s_2s_3s_2 & = & 1432 \\
\Hz\lhd 12|3|4 \lhd 123|4 & t_1t_2 & s_3s_2 & = & 1342 \\
\Hz\lhd 12|3|4 \lhd 123|4 \lhd 1234 & t_1t_2t_4 & s_3 & = & 1243 \\
\Hz\lhd 12|3|4 \lhd 124|3 & t_1t_3 & e & = & 1234 \\
\Hz\lhd 12|3|4 \lhd 124|3 \lhd 1234 & t_1t_3t_4 & s_2 & = & 1324 \\
\Hz\lhd 12|3|4 \lhd 12|34 & t_1t_4 & s_2s_3 & = & 1423 \\
\Hz\lhd 13|2|4 & t_2 & s_1s_3s_2 & = & 3142 \\
\Hz\lhd 13|2|4 \lhd 134|2 & t_2t_4 & s_1s_3 & = & 2143 \\
\Hz\lhd 14|2|3 & t_3 & s_1 & = & 2134 \\
\Hz\lhd 14|2|3 \lhd 134|2 & t_3t_4 & s_1s_2 & = & 3124 \\
\Hz\lhd 1|2|34 & t_4 & s_1s_2s_3 & = & 4123
\end{array}
\]
Now, we draw the Bruhat graph of the interval $[e,w]$
with labelled fat edges forming paths that encode the
decreasing chains $C$.
\begin{center}
$\xymatrix @=0.8cm @C=0.4cm{
& & & 4132 & & & \\
& 1432 \ar@*{[|<2pt>]}@{-}[urr]|{t_1} & & 3142 \ar@*{[|<2pt>]}@{-}[u]|{t_2} & &
4123 \ar@*{[|<2pt>]}@{-}[ull]|{t_4} & \\
1342 \ar@{-}[ur] \ar@*{[|<2pt>]}@{-}[urrr]|(.35){t_1} & &
1423 \ar@{-}[ul] \ar@*{[|<2pt>]}@{-}[urrr]|(.35){t_1} & &
2143 \ar@{-}[ul] \ar@*{[|<2pt>]}@{-}[ur]|(.65){t_2} & &
3124 \ar@{-}[ulll] \ar@*{[|<2pt>]}@{-}[ul]|{t_3} \\
& 1243 \ar@{-}[ul] \ar@{-}[ur] \ar@*{[|<2pt>]}@{-}[urrr]|(.6){t_1} & &
1324 \ar@{-}[ulll] \ar@{-}[ul] \ar@*{[|<2pt>]}@{-}[urrr]|(.4){t_1} & &
2134 \ar@{-}[ul] \ar@{-}[ur] \ar@*{[|<2pt>]}@{-}[uuull]|(.47){t_3} & \\
& & & 1234 \ar@{-}[ull] \ar@{-}[u] \ar@*{[|<2pt>]}@{-}[urr]|{t_1} & & &
}$
\end{center}
By Theorem~\ref{th:injective}, the fat edges form a tree,
and by Proposition~\ref{pr:down}, the fat paths go
down from $w$. By Corollary~\ref{co:surjective},
the fat tree spans all of $[e,w]$.

Assume a chain $C = \{\Hz=X_0 \lhd \dotsb \lhd X_m\}\in \CL$, is such that the smallest hyperplane $H_k$ does not contain $X_m$. Then the chain 
$C_2 = \{X_0 \lhd \dotsb \lhd X_m\lhd (X_m\cap H_k)\}\in \CL$.
This implies that $p(C_2)=p(C)t_k$ and we may thus add $t_k$ from 
the right to any word of descending labels $p(C)$. 
Hence the tree of descending words consists of two isomorphic 
(as edge labelled graphs) copies connected by an edge labelled $t_k$.

Finally, let us relate Theorem~\ref{th:chromatic} to our example.
In the figure above, we see that
$\sum_{u\in[e,w]}q^{a\ell(u,w)}=1+4q+5q^2+2q^3$,
and by Theorem~\ref{th:smoothchromo},
$\chi_{G_w}(t)=(t-1)(t-0)(t-1)(t-2)$.
The reader may check that $\sum_{u\in[e,w]}q^{a\ell(u,w)}
=(-q)^n \chi_{G_w}(-q^{-1})$ as stated in Theorem~\ref{th:chromatic}.
}

\section{Open problems} \label{sec:open}
In this last section, we present some
ideas for future research.
Some of the open problems are intentionally left vague, while others
are more precise.

In Theorem~\ref{th:injective}, we showed that the map
$\phi:\CL \to [e,w]$
is injective for any finite Coxeter group,
but it is not surjective in general.
When the forbidden patterns are avoided, we use
an inductive counting argument showing that
the finite sets $\CL$ and $[e,w]$ have the same cardinality ---
then surjectivity of $\phi$ follows from injectivity.
\begin{open}
Is there a direct proof of
the surjectivity of $\phi$ or, if not, is there another
bijection $\CL \leftrightarrow [e,w]$ whose bijectivity can be proved directly.
\end{open}

\begin{open}
When $\phi$ is not surjective, what is its image?
\end{open}
\noindent
Considering Betti numbers, see Section \ref{S:Betti}, one can deduce that
the number of elements of even length not lying in the image of $\phi$ equals 
the number of such elements of odd length. In particular, evenly many elements
of $[e,w]$ do not lie in the image of $\phi$.

\begin{open}
Find a criterion for the surjectivity of $\phi$ in an arbitrary finite reflection group.
\end{open}

As noted in the introduction, our work (following Postnikov) marks the third appearance of the four patterns
4231, 35142, 42513, and 351624 in the study of flag manifolds and Bruhat order. The first time was in 2002 when Gasharov
and Reiner~\cite{GR} studied the cohomology
of smooth Schubert varieties in partial flag manifolds. In their
paper, they find a simple presentation for the integral cohomology ring,
and it turns out that this presentation holds for a larger class of
subvarieties of partial flag manifold, namely the ones {\em defined
by inclusions}. They characterize these varieties by the same pattern
avoidance condition that apppears in our work.

More recently, Sj\"ostrand~\cite{sjostrand} used
the pattern condition to characterize permutations
whose {\em right hull} covers exactly the lower Bruhat interval
below the permutation; see Lemma~\ref{lem:right_hull}.

As is discussed in~\cite{sjostrand} there seems to be
no direct connection between the ``right hull'' result and the
``defined by inclusions'' result. Though we use Sj\"ostrand's result
in the proof of Lemma~\ref{lem:inductivestep},
we have not found any simple reason
why the same pattern condition turns up again.
\begin{open}
Is there a simple reason why the same pattern condition turns up
in three different contexts: Gasharov and Reiner's
``defined by inclusions'', Sj\"ostrand's ``right hull'', and
Postnikov's (now proved) conjecture?
\end{open}

\begin{open}
Does the poset structure of the Bruhat interval determine the intersection
lattice uniquely? In other words, for any two finite Coxeter systems
$(W,S)$ and $(W',S')$ and elements $w\in W$, $w'\in W'$,
does $[e,w]\cong[e',w']$ imply $L_w\cong L_{w'}$?
\end{open}
\noindent
It is not hard to check that the
assertion is true for $\ell(w)\leq 4$. 

Finally, it would be interesting to know whether our results could
be extended to general Bruhat intervals, i.e.~$[u,w]$ with $u\ne e$.
\begin{open}
Given a (finite) Coxeter system $(W,S)$ and $u,w\in W$ with $u\le w$ in
Bruhat order, is there a hyperplane arrangement $\ca_{u,w}$,
naturally associated with $u$ and $w$, which has as many regions as
there are elements in $[u,w]$ (at least for $u,w$ in some interesting
subset of $W$)?
\end{open}

\end{document}